\documentclass{article}
\usepackage[utf8]{inputenc}
\usepackage{amsmath}
\usepackage{tikz}
\usepackage{amsthm}
\usepackage{amssymb}
\usepackage{mathtools}
\usepackage{booktabs}
\usepackage[toc,page]{appendix}
\usepackage{amsxtra}
\usepackage{amsbsy}
\usepackage{verbatim}
\usepackage{comment}
\usepackage{mathrsfs}
\usepackage{tabularx,ragged2e,booktabs,caption}
\usepackage{longtable}
\usepackage{lipsum}
\usepackage{enumerate}
\usepackage[english]{babel}
\usepackage{blindtext}
\usepackage{url}
\usepackage{bbm}
\usepackage{euscript}
\usepackage{pifont}
\usepackage{hyperref}
\usepackage{wasysym}
\usepackage{fancybox}
\hypersetup{
    colorlinks=true,
    linkcolor=blue,
    filecolor=magenta,
    urlcolor=cyan,
}
\usepackage{graphicx}
\usepackage{epstopdf}

    \usepackage[
    top    = 2.60cm,
    bottom = 2.60cm,
    left   = 2.9 cm,
    right  = 2.9 cm]{geometry}

\newtheorem{thm}{Theorem}[section]
\newtheorem{lem}[thm]{Lemma}
\newtheorem{cor}[thm]{Corollary}
\newtheorem{prop}[thm]{Proposition}

\newtheorem*{conjecture*}{Conjecture}
\newtheorem*{thm*}{Theorem}

\theoremstyle{plain} 
\newcommand{\thistheoremname}{}
\newtheorem*{genericthm}{\thistheoremname}
\newenvironment{namedthm}[1]
  {\renewcommand{\thistheoremname}{#1}%
   \begin{genericthm}}
  {\end{genericthm}}

\theoremstyle{remark}
\newtheorem*{remark}{Remark}

\theoremstyle{definition}
\newtheorem*{define}{Definition}

\newcommand{\N}{\mathbb{N}}
\newcommand{\Z}{\mathbb{Z}}

\newcommand{\Syl}[1]{\mathcal{#1}}

\newcommand{\Sub}{\operatorname{Sub}}

\newcommand{\Dic}{\operatorname{Dic}}

\newskip\aline \newskip\halfaline
\title{\vspace{-1cm}Classifying groups with a small number of subgroups}

\author{Alexander Betz and David A.\ Nash\\}

\date{\today}

\begin{document}
\maketitle
\begin{abstract}
We provide lower bounds on the number of subgroups of a group $G$ as a function of the primes and exponents appearing in the prime factorization of $|G|$.  Using these bounds, we classify all abelian groups with 22 or fewer subgroups, and all non-abelian groups with 19 or fewer subgroups.  This allows us to extend the integer sequence A274847 \cite{OEIS} introduced by Slattery in \cite{Slattery}.
\end{abstract}

It is a classic problem in a first course in group theory to show that a group $G$ has exactly two subgroups if and only if $G \cong \Z_p$ for a prime $p$.  The main idea here is to observe that if $|\Sub G|=2$ or if $G \cong \Z_p$, then every non-identity element $x \in G$ must by necessity generate all of $G$, i.e.\ $\langle x \rangle = G$ for all $x \in G$.  Slightly less frequently, a course may follow up by considering groups $G$ with exactly three or four subgroups.  In those cases, it turns out that we can again argue that $G$ must be cyclic.  

Indeed, if $|\Sub G| = 3$ and $H \leq G$ is the unique, non-trivial, proper subgroup of $G$, then observe that for any $x \in G \setminus H$, we must have $\langle x \rangle = G$ as these elements are non-trivial, must generate a subgroup of $G$, and cannot generate the trivial subgroup or $H$.  Since $G$ must be cyclic, the fact that cyclic groups have exactly one subgroup for each positive divisor of $|G|$ implies that $|G|=p^2$ for some prime $p$ and thus $G \cong \Z_{p^2}$ as cyclic groups of order $|G|$ are unique up to isomorphism.

Similarly, if $|\Sub G| = 4$ and $H \neq K$ are the two non-trivial subgroups, then recall that $H \cup K \leq G$ if and only if $H \leq K$ or $K \leq H$.  It follows that $H \cup K \neq G$ and thus, there exists some $x \in G \setminus (H \cup K)$.  Once again, $\langle x \rangle = G$ and $G$ is cyclic.  As before, a cyclic group $G$ must have exactly one subgroup for each divisor of $|G|$, hence it follows that $G \cong \Z_{pq}$ or $\Z_{p^3}$ for primes $p$ and $q$.  We summarize these classic results below.

\begin{namedthm}{Classic Results} \label{thm: Classic Results}
\leavevmode
\begin{enumerate}
    \item If $|\Sub G| = 2$, then $G \cong \Z_p$ for some prime $p$.
    \item If $|\Sub G| = 3$, then $G \cong \Z_{p^2}$ for some prime $p$.
    \item If $|\Sub G| = 4$, then $G \cong \Z_{pq}$ or $G \cong \Z_{p^3}$ for some primes $p$ and $q$.
\end{enumerate}
\end{namedthm}

These classic results beg the question: Which (necessarily finite) groups $G$ have exactly $k$ subgroups when $k \geq 5$?  Fortunately this becomes much more interesting moving forward as $G$ need not be cyclic when $|\Sub G| \geq 5$.  Thus, from here on we will need a completely different approach.

Miller explored this topic previously in a series of obscure and 
terse papers \cite{Miller1, Miller2, Miller3, Miller4, Miller5} in which he claims to classify the groups with 16 or fewer subgroups, but it is unclear to the authors exactly how he arrives at his conclusions.  Despite that, his results agree with ours, except in the case when $|\Sub G| = 14$ where he seems to have skipped a case, causing him to miss $S_3 \times \Z_3$ and $\Z_3 \rtimes \Z_{32}$.  Given the assertions within, it is certainly clear that Miller is not applying the techniques we use here.  Recently, Slattery \cite{Slattery} explored this idea once more; reducing any group $G$ by factoring out any cyclic central Sylow $p$-subgroups of $G$ first. Using this method, he worked to classify groups with 12 or fewer subgroups up to \emph{similarity} defined in the following sense:

\begin{define}[From \cite{Slattery}]
Let $G$ and $H$ be finite groups. Write $G = P_1 \times P_2 \times \cdots \times P_c \times \widetilde{G}$ and $H = Q_1 \times \cdots \times Q_d \times \widetilde{H}$, where $P_i$ (resp.\ $Q_j$) are cyclic central Sylow subgroups within $G$ (resp.\ $H$).  Then $G$ is \emph{similar} to $H$ if and only if the following conditions hold:
\begin{itemize}
    \item $\widetilde{G}$ is isomorphic to $\widetilde{H}$.
    \item $c = d$
    \item $n_i = m_i$ for some reordering, where $|P_i|=p_i^{n_i}$ and $|Q_i| = q_i^{m_i}$.
\end{itemize}
\end{define}

Using this definition of similarity, groups that are similar will always have the same number of subgroups (see Theorem~\ref{thm:GxHsubgroups} below).  Slattery was encouraged to submit a sequence (A274847 \cite{OEIS}) to the Online Encyclopedia of Integer Sequences  which counts the number of similarity classes of groups with $k$ subgroups.  Unfortunately, while his results for groups with 9 or fewer subgroups agree with ours and those of Miller in \cite{Miller1}, he didn't appear to know about Miller's other papers and has a minor omission in his final table of results.  More specifically, he correctly identifies $D_8$, the dihedral group of order 8, as a non-abelian $p$-group with 10 subgroups, but mistakenly omits it in his final tables causing him to under count the groups with 10 subgroups.  Thus, the 10th term in sequence A274847 should be 12 rather than 11.  As a further point of clarification, in what follows, we stick with GAP\footnote{Groups, Algorithms, and Programming -- see \url{https://www.gap-system.org}.} notation for the groups we list.  As an example, for the \emph{extraspecial} group of order 27, $\langle x,y \mid x^9=y^3=e, yxy^{-1}=x^4\rangle$, which Slattery lists as $E_{27}$, we use the label $M_{27}$ provided by GAP.

Our approach is significantly different from Slattery's as well and takes us significantly further.  In Section~\ref{sec:abelian} we first deal with the case when $G$ is abelian by exploring the number of subgroups of abelian $p$-groups and then considering products of these groups for different primes.  In Section~\ref{sec:nonabelian}, we then approach non-abelian groups using the Sylow Theorems and the Orbit-Stabilizer Theorem to place a lower bound on the number of subgroups of $G$ as a function of the primes and exponents in the prime factorization of $|G|$.  This allows us to greatly reduce the search space for non-abelian groups with 19 or fewer subgroups.  Using the complete lists of similarity classes of abelian and non-abelian groups with 19 or fewer subgroups we then give the first 19 terms in sequence A274847.

\section{Abelian groups}\label{sec:abelian}
Cyclic groups are a straightforward case to begin with as it is well-known that each cyclic group $G$ has exactly one subgroup for each divisor of $|G|$.  Thus, given a cyclic group $G$ of order $|G|=p_1^{a_1}p_2^{a_2} \cdots p_n^{a_n}$, it follows that $|\Sub G| = (a_1+1)(a_2+1)\dots(a_n+1)$.  For one, this implies immediately that there exists at least one group with exactly $k$ subgroups for each $k \in \N$ (namely the group $\Z_{p^{k-1}}$).  In addition, there is exactly one cyclic group of order $|G|$ up to isomorphism, thus we may work backwards to quickly find all cyclic groups with a fixed number of subgroups. 

More generally, by the Fundamental Theorem of Finite Abelian Groups, every such group can be written as a direct product of cyclic groups of prime power orders.  Moreover, for each prime $p$ dividing $|G|$, we may combine the cyclic $p$-groups in the product into a single component subgroup $H_{p^a}$, where $p^a$ is the highest power of $p$ which divides $|G|$.  In this way, we can think of any finite abelian group as a direct product of abelian $p$-groups for different primes $p$.  This is a useful perspective given the following key result:

\begin{thm}\label{thm:GxHsubgroups}
Let $G$ and $H$ be groups. If $(|G|,|H|)=1$ then $|\Sub G \times H| = |\Sub G| \cdot |\Sub H|$
\end{thm}

\begin{proof}
Certainly $G' \times H' \leq G \times H$ for all $G' \leq G$ and $H' \leq H$, so it suffices to show that every subgroup $K \leq G \times H$ can be split as $K = G' \times H'$ for some $G'$ and $H'$. Observe, since $G \times H =\{ (g,h) \mid g \in G, h \in H\}$, if $K \leq G \times H$, then we may define
$$K_G=\{ g \in G \mid (g,h) \in K \text{~for some~} h \in H \}$$
$$K_H=\{ h \in H \mid (g,h) \in K \text{~for some~} g \in G \}.$$ Our goal is to show that $K = K_G \times K_H$. By our assumption, and Lagrange's Theorem, we know $o(g) \mid |G|$ and $(o(g),o(h))=1~ \forall g \in G, h \in H$.  Consider $g \in K_G$ with $(g,h_g) \in K$.  Since $g$ and $h_g$ have coprime orders, it follows that $\langle (g,h_g) \rangle$ will be the cyclic group $\Z_{o(g)o(h_g)}$. Since $\langle (g,h_g) \rangle \leq K$, it follows that $(g,e_H) \in K ~ \forall g \in K_G$.  A similar argument will show that $(e_G,h) \in K ~ \forall h \in K_H$ as well. 

Since $K \leq G \times H$, by closure we have $(g,h) \in K ~ \forall g \in K_G,~ h \in K_H$.
It follows that $K_G \times K_H \subseteq K$ and thus, since $K \subseteq K_G \times K_H$ by definition, we have $K = K_G \times K_H$.  The above argument also shows that $K_G \cong K \cap (G \times \{e_H\}$), thus $K_G \leq G$.  Similarly, $K_H \leq H$, which completes the proof.
\end{proof}

Given this result and the perspective above, we may count the number of subgroups of any finite abelian group as long as we know the number of subgroups of its abelian $p$-group components.

\begin{cor} \label{cor: abelian p prime subgroups}
If $G$ is an abelian group as defined above with $|G| = p_1^{a_1}p_2^{a_2}\cdots p_n^{a_n}$ then $|\Sub G| = |\Sub H_{p_1^{a_1}}| \cdot |\Sub H_{p_2^{a_2}}| \cdots |\Sub H_{p_n^{a_n}}|$. Furthermore this implies that if an abelian group $G$ has a prime number of subgroups then $G$ must be a $p$-group.
\end{cor}

This leaves us to describe the number of subgroups in abelian $p$-groups.  The case when an abelian $p$-group has exactly two cyclic factors has been fully described by Ali and Al-Awami \cite{Abelian Sub thesis}.\footnote{Note: The statement of Theorem~4.2.1 in \cite{Abelian Sub thesis} has a typo, but their proof proves the statement given here and \cite{MT Abelian Sub} confirms this result.}

\begin{thm}[Theorem~4.2.1 in \cite{Abelian Sub thesis}] \label{prop: ZxZ class}~\\
If $G \cong \Z_{p^a} \times \Z_{p^b}$ where $a \leq b$ then the number of subgroups of $G$ satisfies \[|\Sub G| = \frac{1}{(p-1)^2} \left[(b-a+1)p^{a+2}-(b-a-1)p^{a+1}-(b+a+3)p+(b+a+1) \right] \]
\end{thm}

In addition to this powerful result, we also wish to consider abelian $p$-groups with more than two factors.  We therefore address a few special cases which will be enough for our purposes.

\begin{prop} \label{prop: 3ZP}
If $G \cong (\Z_p)^n = \Z_p \times \cdots \times \Z_p$, then 
$$|\Sub G| = \sum_{i=0}^n {{n}\choose{i}}_p, \text{~where~} {{n}\choose{i}}_p = \frac{(1-p^n)(1-p^{n-1})\cdots(1-p^{n-i+1})}{(1-p)(1-p^2)\cdots(1-p^i)}.$$
\end{prop}
\begin{proof} 
The group $(\Z_p)^n$ is also an $n$-dimensional vector space over $\Z_p$ and each subgroup of order $p^i$ in $G$ corresponds to an $i$-dimensional subspace.  It is well-known that the Gaussian binomial coefficient ${{n}\choose{i}}_p$ counts the number of $i$-dimensional subspaces of an $n$-dimensional vector space over $\Z_p$.
\end{proof}

In addition to these specific cases, a recent paper by Aivazidis and M\"uller \cite{AM} gives more general lower bounds on the number of subgroups in non-cyclic $p$-groups.\footnote{We thank Aivazidis for bringing this to our attention.  In an earlier version we proved similar results directly.}  For example, the results below imply that any non-cyclic abelian $p$-group with fewer than 23 subgroups must have order $p^a$ for $a \leq 7$.

\begin{thm}[Theorem A from \cite{AM}]\label{cor: 3+-group bound}
Let $G$ be a non-cyclic group with $|G|=p^a$ for a prime $p \geq 3$.  Then $|\Sub G| \geq (a-1)(p+1)+2$, with equality if and only if $G \cong \Z_{p^{a-1}} \times \Z_p$ or $G \cong M_{p^a} = \langle x,y \mid x^{p^{a-1}} = y^p = e, yxy^{-1} = x^{1+p^{a-2}}$.
\end{thm}

\begin{thm}[Theorem B from \cite{AM}]\label{thm:p-group bound}
Let $G$ be a non-cyclic group with $|G|=2^a$.  If $a=3$, then $|\Sub G| \geq 6$ with equality if and only if $G \cong Q_8$.  And if $a\geq 4$, then $|\Sub G| \geq 3a-1$ with equality if and only if $G\cong Q_{16}$ or $G \cong \Z_{2^{a-1}} \times \Z_2$, or $M_{2^{a}} = \langle x,y \mid x^{2^{a-1}}=y^2=e, yxy^{-1} = x^{1+2^{a-2}}\rangle$.
\end{thm}

Classifying abelian groups with exactly $k$ subgroups is now a matter of finding all possible ways to combine abelian $p$-groups for different primes so that the product of their individual numbers of subgroups equals $k$.  Recall also that, thanks to Theorem~\ref{thm:GxHsubgroups}, an abelian group can only have a prime number of subgroups if it is a $p$-group.

As an example to demonstrate, suppose we wish to find all abelian groups with exactly 10 subgroups. We must consider abelian $p$-groups with exactly 10 subgroups themselves, or a product of an abelian $p$-group with an abelian $q$-group such that one has 5 subgroups and the other has 2 subgroups.  Applying Theorem~\ref{prop: ZxZ class}, we find that the only abelian groups (up to similarity) with 10 subgroups are $\Z_{p^9}$, $\Z_{p^4q}$, $\Z_2 \times \Z_2 \times \Z_p$ ($p \neq 2$), $\Z_7 \times \Z_7$, and $\Z_9 \times \Z_3$.  Continuing in this manner, Table~\ref{tab:Abelian} reports all similarity classes of abelian groups with fewer than 23 subgroups.  Note that, just as in the case of $\Z_2 \times \Z_2 \times \Z_p$, arbitrary primes are always assumed to be relatively prime to any others appearing.

\begin{table}[h!]
    \centering
    \begin{tabular}{c|c|c}
         $|\Sub G|$ & Similarity Classes of Groups & \# of Classes \\ \hline
         $1$&$\{e\}$ & 1   \\  \hline
        $2$ & $\Z_p$ & 1 \\  \hline
         $3$& $\Z_{p^2}$ & 1  \\ \hline
        $4$ & $\Z_{p^3}$, $\Z_{pq}$ & 2   \\  \hline
         $5$& $\Z_{p^4}$, $\Z_2 \times \Z_2$ & 2  \\  \hline
        $6$ & $\Z_{p^5}$, $\Z_{p^2q}$, $\Z_3 \times \Z_3$ & 3 \\  \hline
        $7$& $\Z_{p^6}$ & 1   \\  \hline
        $8$ & $\Z_{p^7}$, $\Z_{p^3q}$, $\Z_{pqr}$, $\Z_4 \times \Z_2$, $\Z_5 \times \Z_5$ & 5  \\  \hline
         $9$& $\Z_{p^8}$, $\Z_{p^2q^2}$ & 2   \\ \hline
        $10$ & $\Z_{p^9}$, $\Z_{p^4q}$, $\Z_2 \times \Z_2 \times \Z_p$, $\Z_9 \times \Z_3$, $\Z_7 \times \Z_7$ & 5  \\  \hline
         $11$& $\Z_{p^{10}}$, $\Z_8 \times \Z_2$ & 2  \\  \hline
        $12$ & $\Z_{p^{11}}$, $\Z_{p^5q}$, $\Z_{p^3q^2}$, $\Z_{p^2qr}$, $\Z_3 \times \Z_3 \times \Z_p$ & 5  \\  \hline
         $13$ & $\Z_{p^{12}}$ & 1  \\  \hline
         $14$& $\Z_{p^{13}}$, $\Z_{p^6q}$, $\Z_{16}\times \Z_2$, $\Z_{27} \times \Z_3$, $\Z_{25}\times \Z_5$, $\Z_{11}\times \Z_{11}$ & 6   \\  \hline
        $15$ &  $\Z_{p^{14}}$, $\Z_{p^4q^2}$, $\Z_4 \times \Z_4$, $\Z_2 \times \Z_2 \times \Z_{p^2}$ & 4  \\  \hline
        & $\Z_{p^{15}}$, $\Z_{p^7q}$, $\Z_{p^3q^3}$, $\Z_{p^3qr}$, $\Z_{pqrs}$, $\Z_2 \times \Z_2 \times \Z_2$, & \\
        16 & $\Z_4 \times \Z_2 \times \Z_p$, $\Z_5 \times \Z_5 \times \Z_p$,  $\Z_{13} \times \Z_{13}$ & 9  \\ \hline
        17 & $\Z_{p^{16}}$, $\Z_{32} \times \Z_2$ & 2 \\ \hline
        18 & $\Z_{p^{17}}$, $\Z_{p^8q}$, $\Z_{p^5q^2}$, $\Z_{p^2q^2r}$, $\Z_{81} \times \Z_3$, $\Z_3 \times \Z_3 \times \Z_{p^2}$, $\Z_{49} \times \Z_7$ & 7 \\ \hline
        19 & $\Z_{p^{18}}$ & 1 \\ \hline
        & $\Z_{p^{19}}$, $\Z_{p^9q}$, $\Z_{p^4q^3}$, $\Z_{p^4qr}$, $\Z_{64} \times \Z_2$, $\Z_2 \times \Z_2 \times \Z_{p^3}$, $\Z_2 \times \Z_2 \times \Z_{pq}$ &\\
        20 &  $\Z_9 \times \Z_3 \times \Z_p$, $\Z_{125} \times \Z_5$, $\Z_7 \times \Z_7 \times \Z_p$, $\Z_{17} \times \Z_{17}$ & 11  \\ \hline
        21 & $\Z_{p^{20}}$, $\Z_{p^6q^2}$ & 2 \\ \hline
        22 & $\Z_{p^{21}}$, $\Z_{p^{10}q}$, $\Z_8 \times \Z_4$, $\Z_{8} \times \Z_2 \times \Z_p$, $\Z_{243} \times \Z_3$, $\Z_{19} \times \Z_{19}$ & 6
    \end{tabular}
    \caption{Similarity classes of abelian groups with fewer than 23 subgroups.}
    \label{tab:Abelian}
\end{table}

One could continue to extend this process quite a bit further without running into too much resistance.\footnote{In fact, at the request of OEIS we have \cite{BN}.}  The much more interesting and challenging case lies in discussion of non-abelian groups.

\section{Non-abelian groups} \label{sec:nonabelian}
For non-abelian groups, it would be nice if we could apply the same sorts of techniques to count the number of subgroups by understanding smaller components.  The best analog available for decomposing a group into relatively prime $p$-group parts is the collection of Sylow Theorems.

\begin{namedthm}{The Sylow Theorems}\label{SYL-THM}
 Let $G$ be a finite group, $p$ a prime divisor of $|G|$ and write $|G| = p^at$ where $a$ and $t$ are positive integers, and $p$ does not divide $t$.  Let $\operatorname{Syl}_p(G) = \{\Syl{P} \leq G \mid |\Syl{P}|=p^a\}$.
\begin{enumerate}[I]
    \item{There exists a \emph{Sylow $p$-subgroup} $\Syl{P} \in \operatorname{Syl}_p(G)$.}
    \item{If $\Syl{P}, \Syl{P'} \in \operatorname{Syl}_p(G)$, then there exists a $g \in G$ with $\Syl{P'} = g\Syl{P}g^{-1}$.}
    \item{Let $n_p = |\operatorname{Syl}_p(G)|$. Then $n_p\mid t$ and $n_p = [G:N(\Syl{P})] \equiv 1 \pmod{p}$, where $N(\Syl{P})$ is the normalizer.}
\end{enumerate}
\end{namedthm}

These famous results give us some information regarding the number of $p$-subgroups within a group $G$, but they do not directly tell us about how the different $p$-group components will interact with one another.  If $G$ is especially nice -- i.e.\ if each of its Sylow subgroups is unique and normal -- then $G$ will decompose as a direct product of its Sylow subgroups (see e.g.\ Corollary 5.4.2 in \cite{Nash}) and we may apply Theorem~\ref{thm:GxHsubgroups} to count the number of subgroups directly.  Unfortunately, this is frequently not the case when $G$ is non-abelian.  In addition, the Sylow $p$-subgroups themselves need not be abelian, thus we need a way to explore the subgroups of non-abelian $p$-groups as well.

As the Sylow Theorems lend themselves to breaking a group into $p$-group components, they do not give us much information in the case when $G$ is itself a $p$-group (in which case $G = \Syl{P}$ and $n_p=1$).  Thankfully there is a generalization of Sylow (III) due to Wielandt \cite{Wielandt} which places conditions on the number of $p$-subgroups for each power of $p$.

\begin{thm}[From \cite{Wielandt}\footnote{For a more modern treatment in English, see \url{https://people.bath.ac.uk/dmjc20/GpThy/wiel.pdf}.}] \label{thm: Weilandt} 
Let $G$ be a group with $|G|=p^a$ for a prime $p$. Then the number of subgroups of order $p^i$ $(i \leq a)$ is equivalent to $1 \pmod{p}$.
\end{thm}

In addition, Theorem~\ref{cor: 3+-group bound} and Theorem~\ref{thm:p-group bound} also provide lower bounds on the number of subgroups of non-abelian $p$-groups.  With these in hand, we move to non-abelian groups whose orders are divisible by multiple primes.

\subsection{Non-abelian groups with $|G|$ divisible by multiple primes}
When $|G|$ is divisible by multiple primes, the Sylow Theorems allow us to explore the $p$-group components for each prime $p$ in the prime factorization of $|G|$, however it is unclear exactly how those components will interact with one another.  In the nicest situation, when $G$ is nilpotent, then $G$ can be written as a direct product of its Sylow subgroups\footnote{Note that this situation exactly corresponds to $G$ having only \emph{normal} Sylow subgroups.} and we may apply Theorem~\ref{thm:GxHsubgroups} to directly count the number of subgroups.  Slightly more generally, whenever $G$ can be expressed as a direct product of subgroups with coprime orders then this avenue will be available to us -- i.e.\ we can find all such groups with exactly $k$ subgroups by exploring the ways to factor $k$ (having already understood groups with fewer than $k$ subgroups).

When this is not the case, we need a different way to count subgroups.  The Sylow Theorems and Theorem~\ref{thm: Weilandt} provide some information about the number of $p$-subgroups within $G$, but we need to be able to count subgroups of composite orders as well.  Thankfully (see e.g.\ Proposition~4.2.11 in \cite{Nash}), whenever $N \unlhd G$ and $H \leq G$, then the set product  $NH \leq G$ and in fact, $NH \unlhd G$ if $H \unlhd G$ too.  Moreover, if $N$ and $H$ have relatively prime orders, then we must have $N \cap H = \{e\}$ and it follows that $|NH| = |N|\cdot|H|$.

This will allow us to use normal $p$-subgroups, together with $q$-subgroups, to create subgroups of composite orders.  To demonstrate the effectiveness of this idea through an example, we need to set up some notation.  Given a fixed group $G$, in what follows we will use $H_n$ for $n \in \N$ to denote a subgroup of order $n$ in $G$. 
Now suppose that $p^a$ and $q^b$ are highest powers of primes $p \neq q$ which divide $|G|$.  There must exist at least one subgroup $H_{p^i} \leq G$ for each $1 \leq i \leq a$ and similarly, we have at least one $H_{q^j} \leq G$ for each $0 \leq j \leq b$ (this second collection includes the trivial subgroup $\{e\}$)
.  Now, for each such prime power $p^i$, observe that Theorem~\ref{thm: Weilandt} implies that either $H_{p^i}$ is unique -- in which case we can create at least $b+1$ distinct product subgroups $H_{p^i}H_{q^j}$ for each $0 \leq j \leq b$ (including $H_{p^i}$ itself) -- or $H_{p^i}$ is not unique, in which case there must be at least $p+1$ subgroups of order $p^i$.  Running through the $a$ different prime powers, we have demonstrated the existence of at least $a\cdot \min(b+1,p+1)$ distinct subgroups in $G$ (possibly including $G$ itself).  Note that, in some situations it will be helpful to treat the Sylow subgroups themselves separately (i.e.\ not allowing $i=a$ or $j=b$).

In what follows, we consider different cases based on the number of distinct primes which divide $|G|$.  In each case, we demonstrate a lower bound on the number of subgroups of non-nilpotent $G$ as a function of the primes and exponents in the prime factorization of $|G|$ thereby reducing the search space to a small finite number of cases.

Beginning with $|G|$ being divisible by only two primes $p<q$, it is helpful to recall (see e.g.\ \cite{Nash}) that there exists a non-nilpotent group $G$ of order $pq$ if and only if $q \equiv 1 \pmod{p}$.  Moreover, since subgroups of index $p$, where $p$ is the smallest prime dividing $|G|$ must be normal (see Theorem 1 in \cite{Lam}), it follows that $\Syl{Q}$ must be normal and thus Sylow (III) implies that $|\Sub G| = q+3$ for such a group.  
For the more general situation, we recall a classic result of Burnside:

\begin{lem}[From \cite{Burnside}]\label{lem:Burnside}
Let $G$ be a group and let $p$ be the smallest prime dividing $|G|$.  If $\Syl{P} \in \operatorname{Syl}_p(G)$ is cyclic, then $\Syl{P}$ has a normal complement.
\end{lem}

It follows immediately that if $|G|$ is divisible by exactly two primes $p<q$ and $\Syl{P} \in \operatorname{Syl}_p(G)$ is cyclic, then $n_q=1$ as the complement of $\Syl{P}$ is a Sylow $q$-subgroup. In addition, it is well-known that $G$ is nilpotent if and only if every maximal subgroup of $G$ is normal  Thus, if $G$ is non-nilpotent it follows that $n_p \neq 1$ and $G$ must contain at least one non-normal maximal subgroup.  When $G$ has cyclic Sylow subgroups though, the maximal subgroups of $G$ have prime index.  

Indeed, let $|G|=p^aq^b$ with $\langle x \rangle = \Syl{P} \in \operatorname{Syl}_p(G)$ and $\langle y \rangle = \Syl{Q} \in \operatorname{Syl}_q(G)$ both cyclic and let $H \leq G$ be a subgroup with $|H|=p^iq^j$ for $i<a$ and $j<b$.  Since $\Syl{Q} \unlhd G$, then $H\Syl{Q}$ is a proper subgroup which contains $H$ and has order $p^iq^b$, hence $H$ cannot be maximal.  Similarly, if $|H|=p^aq^j$ for $j<b-1$, then $H = \Syl{P}\langle y^{q^{b-j}}\rangle$ for some Sylow $p$-subgroup.  It follows that the product of $H$ with the unique normal subgroup $<y^q>$ of order $q^{b-1}$ will create a proper subgroup containing $H$ -- again implying that $H$ cannot be normal.  Thus, maximal subgroups of $G$ must have order $p^{a-1}q^b$ or $p^aq^{b-1}$.

Since subgroups of index $p$ must be normal, it follows that if $G$ is non-nilpotent then it must contain a non-normal subgroup of order $p^aq^{b-1}$.  Hence, it contains exactly $q$ such subgroups.  Moreover, applying the Orbit-Stabilizer Theorem, this implies that $G$ must also contain (at least $q$) non-normal subgroups of order $p^aq^j$ for each $1 \leq j < b$.  With these observations, we are now ready to describe bounds on $|\Sub G|$ when $|G|$ is divisible by exactly two distinct primes.

\begin{thm}\label{thm:p^aq^b}
Let $G$ be non-nilpotent with $|G|=p^aq^b$ for primes $p<q$ 
, then
$$|\Sub G| \geq \min\{bq+ab+a+1,~ b(q+1) + 2a + (a-1)\min(p,b),~ b+q+1+(b-1)\min(a,q) + m_p\},$$
where we only consider the second term if $b>1$, and we only consider the third term if $a>1$, in which case $m_2 = 5$ if $a=2$, $m_2=6$ if $a=3$, and $m_p = (a-1)(p+1)+2$ if $p \neq 2$ or $a \geq 4$.
\end{thm}
\begin{proof}
We consider two main cases based on whether $\Syl{P} \in \operatorname{Syl}_p(G)$ is cyclic or not.

(1) If $\Syl{P}$ is cyclic, then $\Syl{P} {\not}{\unlhd} G$ and $\Syl{Q} \in \operatorname{Syl}_q(G)$ is normal as we saw above.  (i) If $\Syl{Q} \in \operatorname{Syl}_q(G)$ is cyclic too, then by the previous discussion we have at least $bq$ subgroups of orders $p^aq^j$ ($0 \leq j < b$) -- including the at least $q$ Sylow $p$-subgroups.  Moreover, every $q$-subgroup is unique and normal, thus we also have at least $ab$ subgroups of orders $p^iq^j$ ($0 \leq i < a$, $1 \leq j \leq b$).  Finally, we also have at least $a-1$ distinct $p$-subgroups of lower orders.  Together with $G$ and $\{e\}$, this shows $|\Sub G| \geq bq + ab + a + 1$.

(ii) If instead, $\Syl{Q}$ is not cyclic, then $b>1$ and by Theorem~\ref{cor: 3+-group bound}, $\Syl{Q}$ must have at least $(b-1)(q+1)+2$ subgroups (including $\{e\}$).  Since $\Syl{Q} \unlhd G$, we still have at least $a$ subgroups of orders $p^iq^b$ ($1 \leq i \leq a$) including $G$ itself.  In addition to the at least $q$ Sylow $p$-subgroups, we must also have at least $(a-1)\min(p+1,b+1)$ subgroups of orders $p^iq^j$ ($1 \leq i < a$, $0 \leq j < b$).  All together, this shows $|\Sub G| \geq (b-1)(q+1)+2 + a + q + (a-1)\min(p+1,b+1) = b(q+1)+2a+(a-1)\min(p,b)$.

(2) If $\Syl{P}$ is not cyclic, then $a>1$ and we may apply Theorem~\ref{thm:p-group bound} or Theorem~\ref{prop: ZxZ class} (subtracting 1) to count the \emph{proper} subgroups of $\Syl{P}$.  Since $G$ is non-nilpotent, we must have at least one non-normal Sylow subgroup.  Observe that if $\Syl{Q} \unlhd G$, then there must be at least $q+1$ total Sylow subgroups, together with at least one product $H_{pq^b}$ (since $a>1$) -- and if $\Syl{Q} {\not}{\unlhd} G$, then there must be at least $q+2$ Sylow subgroups total.  Moreover, for each of the $q$-subgroups of lower order, either they are unique and we may create product subgroups with subgroups of $\Syl{P}$, or there are at least $q+1$ of them by Theorem~\ref{thm: Weilandt}.  It follows that there must be at least $(b-1)\min(a+1,q+1)$ subgroups of orders $p^iq^j$ ($0 \leq i \leq a$, $1 \leq j < b$).  Counting $G$ makes up for the 1 we subtracted from the subgroups of $\Syl{P}$, thus we have demonstrated that $|\Sub G| \geq q+2 + (b-1)\min(a+1,q+1) + m_p = b+q+1+(b-1)\min(a,q) + m_p$, where $m_p$ is the number of subgroups of $\Syl{P}$ coming from Theorem~\ref{thm:p-group bound} or Theorem~\ref{prop: ZxZ class} as applicable.
\end{proof}

Note that the above bound is sharp when $b=1$. 
Next, consider the situation when $|G|$ is divisible by three primes $p<q<r$.  Recall that whenever $G$ can be decomposed as a non-trivial direct product of subgroups with coprime orders, then we can apply Theorem~\ref{thm:GxHsubgroups}, thus we again consider only the groups which cannot be decomposed in this way.  

\begin{thm}\label{thm:p^aq^br^c}
Let $G$ be a non-nilpotent group that cannot be decomposed as a non-trivial direct product of two groups with coprime orders.  If $|G|=p^aq^br^c$ 
with $p<q<r$ 
then,
\begin{equation}\label{eq:p^aq^br^c}
\begin{split}
|\Sub G| \geq a+b+c + (a-1)\min(b+1,p) + (b-1)\min(c+1,q) + (c-1)\min(a+1,r)+\\
\min\{p+q+r+2,~ q+2+\min(p^i+1,q^j,2r+1
),~ \min(r+1,2q+2
) + \min(r+1,2q)\},
\end{split}
\end{equation}
where $i$ and $j$ are minimal such that $p^i, q^j \geq r+1$.
\end{thm}
\begin{proof} Recall that since $G$ is non-nilpotent, there must be at least one non-normal Sylow subgroup.  First we explore the Sylow subgroups of $G$ and their potential products by considering the number of Sylow subgroups which are normal in $G$.  The most straightforward situation is when all three are non-normal, in which case, by Sylow (III), there must be at least $2q+r+2$ distinct Sylow subgroups in $G$ (and potentially no product subgroups involving two different Sylows).

If instead exactly two of the Sylow subgroups are normal, then there exist product subgroups $H_{p^aq^b}$, $H_{p^ar^c}$, and $H_{q^br^c}$. The two constructed as products with the lone \emph{non-normal} Sylow subgroup cannot be normal themselves however, as that would imply that $G$ could be decomposed into a direct product of groups with coprime orders.  Since those product subgroups are self-stabilizing, the Orbit-Stabilizer Theorem implies that the size of their orbits under conjugation must divide whatever prime power is missing in their order.  For example, if $\Syl{R} \in \operatorname{Syl}_r(G)$ is the non-normal one, then we have at least $p+q$ subgroups of orders $p^ar^c$ or $q^br^c$ in addition to $r+3$ total Sylow subgroups for at least $p+q+r+3$ subgroups.  Similarly, if $\Syl{Q} \in \operatorname{Syl}_q(G)$ is non-normal, this count becomes $p+r+(q+3)$, while if $\Syl{P} \in \operatorname{Syl}_p(G)$ is non-normal, it is $q+r+(q+2)$ (which is $\geq p+q+r+3$ for all primes $p<q<r$). 

The most delicate case is when only one Sylow subgroup is normal, and we proceed based on whether $\Syl{R} \unlhd G$ or not. (i)  If $\Syl{R}$ is not normal, then $n_r \geq r+1$ by Sylow (III).  Moreover, if $\Syl{P} \unlhd G$ (resp. $\Syl{Q} \unlhd G$), then $n_p + n_q \geq q+2$ (resp. $q+1$).  In addition, we have product subgroups $H_{p^aq^b}$ and $H_{p^ar^c}$ (resp. $H_{q^br^c}$).  However, if any subgroup of order $p^ar^c$ (resp. $q^br^c$) contains multiple Sylow $r$-subgroups, then $n_r \geq p^i$ for some $i$ (resp. $n_r \geq q^j$ for some $j$) as well.   And if none of those subgroups contain multiple Sylow $r$-subgroups, then there must be at least $n_r \geq r+1$ of them.  Thus, counting all Sylow subgroups and products of them, we must have at least $q+\min(p^i+4,q^j+3,2r+4 
) = q+3+\min(p^i+1,q^j,2r+1)$, where $i$ and $j$ are minimal such that $p^i, q^j \geq r+1$.

(ii) If $\Syl{R}$ is normal, then there exist product subgroups $H_{p^ar^c}$ and $H_{q^br^c}$. 
As in the previous case, if any subgroup of order $p^ar^c$ contains multiple Sylow $p$-subgroups, then $n_p \geq r$, and if none do, then there must be exactly $n_p \geq q$ subgroups of order $p^ar^c$. Summarizing, for subgroups of order $p^a$ or $p^ar^c$, there must be at least $\min(r+1,2q)$.  Repeating the argument, if any subgroup of order $q^br^c$ contains multiple Sylow $q$-subgroups, then $n_q \geq r$, and if not, then there must be exactly $n_q \geq q+1$ subgroups of order $q^br^c$.  Summarizing, there must be least $\min(r+1,2q+2)$ 
subgroups of orders $q^b$ or $q^br^c$
.  In total, if exactly one Sylow subgroup is normal, then $G$ must contain at least $\min\{q+3+\min(p^i+1,q^j, 2r+1),~ \min(r+1,2q+2) + \min(r+1,2q) + 1\}$ subgroups of orders $p^a$, $q^b$, $r^c$, $p^aq^b$, $p^ar^c$, or $q^br^c$.

For the rest of the subgroups, starting with the prime powers $p^i$, $1 \leq i < a$, as before these account for at least $(a-1)\min(b+2,p+1)$ subgroups (by pairing them up with the $q$-subgroups and $\Syl{R}$).  Similarly, each prime power $q^j$, $1 \leq j < b$ accounts for at least $(b-1)\min(c+2,q+1)$ subgroups (by pairing them up with the $r$-subgroups and $\Syl{P}$) and each power $r^\ell$, $1 \leq \ell < c$ accounts for $(c-1)\min(a+2,r+1)$ subgroups (by pairing them up with the $p$-subgroups and $\Syl{Q}$).  Together with $G$ and $\{e\}$, this is at least $(a-1)\min(b+2,p+1) + (b-1)\min(c+2,q+1) + (c-1)\min(a+2,r+1) + 2$ additional subgroups.

Furthermore, since $p+q+r+3 \leq 2q+r+2$ for all primes $p<q<r$, with some minor arithmetic simplifications we have shown that the inequality in (\ref{eq:p^aq^br^c}) holds.
\end{proof}

\begin{remark}  Note, we could have taken multiple perspectives when counting the subgroups of lower orders (i.e.\ by pairing them up in different ways).  Since all of these perspectives would be valid, the lower bound would be the maximum of each. However, for our purposes, this particular choice (which exhibits some level of symmetry in $a$, $b$, and $c$) was good enough.  Indeed, as soon as at least one of $a$, $b$, or $c$ is greater than 1, our bound shows that $|\Sub G| \geq 18$ when $r=5$ and $|\Sub G| \geq 20$ when $r\geq 7$.
\end{remark}

We now consider the special case when $|G|=pqr$ so that we may improve our bound slightly.

\begin{thm}\label{thm:pqr}
Let $G$ be non-nilpotent with $|G|=pqr$ $(p<q<r)$ that cannot be decomposed as a non-trivial product of subgroups with coprime orders.  Then $|\Sub G| \geq r+4 + \min(r+1, 2q)$.
\end{thm}
\begin{proof}
It is known (see e.g.\ \cite{Cole}) in a group of order $pqr$ with $p<q<r$, that $\Syl{R} \in \operatorname{Syl}_r(G)$ must be a normal subgroup.  Hence there exist product subgroups $H_{pr}$ and $H_{qr}$, the latter of which must be normal as it has index $p$ (again, see Theorem 1 in \cite{Lam}).  Since $G$ cannot be decomposed into co-prime parts, it follows that $\Syl{P} \in \operatorname{Syl}_p(G)$ cannot be normal, as otherwise $G \cong \Syl{P} \times H_{qr}$.  From here, we will consider two cases based on whether $\Syl{Q} \in \operatorname{Syl}_q(G)$ is normal.

(1) If $\Syl{Q}$ is normal, then there exists a product subgroup $H_{pq}$ which cannot be normal (otherwise $G$ could be decomposed), thus there must be $r$ such subgroups by the Orbit-Stabilizer Theorem.  Note however, that with $\Syl{Q}$ unique, there can only be $r$ subgroups of order $pq$ when $n_p \geq r$.  In addition, the subgroup $H_{pr}$ cannot be normal either, thus there must be $q$ subgroups of order $pr$.  Together with $\Syl{Q}$, $\Syl{R}$, $H_{qr}$, $G$, and $\{e\}$ this shows that $|\Sub G| \geq q+2r+5$.

(2) If instead, $\Syl{Q}$ is not normal, then the normal subgroup $H_{qr}$ must contain all of them -- hence $n_q = r$ by Sylow (III).  In addition, since the product of any Sylow $p$-subgroup with $\Syl{R}$ will result in a group of order $pr$, it follows that either $H_{pr}$ contains all of the Sylow $p$-subgroups -- meaning $n_p = r$ -- or there must be $q$ subgroups of order $pr$ by the Orbit-Stabilizer Theorem.  Together with $\Syl{R}$, $G$, and $\{e\}$ this shows that $|\Sub G| \geq r+4 + \min(r+1, 2q)$.  Note that this is strictly less than the expression given in case (1).
\end{proof}

Together, Theorem~\ref{thm:p^aq^br^c} and Theorem~\ref{thm:pqr} imply that if there exists a non-nilpotent group $G$, with $|G|$ divisible by three primes, which cannot be decomposed as a direct product of subgroups with coprime orders, and fewer than 20 subgroups, then $|G|=30$, 42, 60, 90, or 150.  Next, we rule out all non-nilpotent groups whose orders are divisible by four or more primes.  The argument is much more elementary than in the three prime case as there are many more products involving only the Sylow subgroups themselves

\begin{thm}\label{thm:4primes}
Let $G$ be a non-nilpotent group with $|G|=p^aq^br^cs^dt$ for primes $p<q<r<s$, where $t \in \N$ is relatively prime to $pqrs$.  Then $|\Sub G| \geq 20$.
\end{thm}
\begin{proof}
We will break this up into 5 cases corresponding to the number of normal Sylow subgroups from within the collection of only $\Syl{P} \in \operatorname{Syl}_p(G)$, $\Syl{Q} \in \operatorname{Syl}_q(G)$, $\Syl{R} \in \operatorname{Syl}_r(G)$, and $\Syl{S} \in \operatorname{Syl}_s(G)$.  

(1) Suppose all four Sylow subgroups are normal. Since the product of normal subgroups is again normal, we can actually create 10 product subgroups $H_{p^aq^b}$, $H_{p^ar^c}$, $H_{p^as^d}$, $H_{q^br^c}$, $H_{q^bs^d}$, $H_{r^cs^d}$, $H_{p^aq^br^c}$, $H_{p^aq^bs^d}$, $H_{p^ar^cs^d}$, and $H_{q^br^cs^d}$.  In addition, we have the four Sylow subgroups themselves and $G$ and $\{e\}$ for at least 16 subgroups.  However, the fact that $G$ is not nilpotent implies that $t \neq 1$ and thus, there must exist at least one other Sylow subgroup for a prime dividing $t$ and at least 10 additional product subgroups.

(2) Suppose that there exists one non-normal Sylow subgroup.  The 10 product subgroups described in case (1) must all still exist, as well as three normal Sylow subgroups, $G$, and $\{e\}$. If $\Syl{S} {\not} {\unlhd} G$, then $n_s \geq s+1$ -- already giving us $16+s \geq 23$ subgroups.  If instead $\Syl{R} {\not} {\unlhd} G$, then $n_r \geq r+1$ -- already giving us $16+r \geq 21$.  If $\Syl{Q} {\not} {\unlhd} G$, then either a single product $H_{q^bs^d}$ contains multiple Sylow $q$-subgroups and $n_q \geq s$, or there are at least $p$ subgroups of order $q^bs^d$ -- this already gives us at least $15 + \min(s+1,p+q+1) \geq 21$.  Finally, if $\Syl{P} {\not} {\unlhd} G$, then either a single product $H_{p^as^d}$ contains multiple Sylow $p$-subgroups and $n_p \geq s$, or there are at least $q$ subgroups of order $p^as^d$ -- already giving us $15 + \min(s+1, 2q) \geq 21$.

(3) Suppose two of them are not normal, say $\Syl{P}$ and $\Syl{Q}$.  Then the 7 product subgroups that do not involve \emph{both} $p$ and $q$ must still exist.  Note however, that either $H_{p^ar^c}$ is normal -- and we can also create $H_{p^aq^br^c}$ -- or it is not normal, in which case there are multiple of them.  Thus, the product $p^ar^c$ must account for at least two subgroups.  A similar argument shows that $p^as^d$ must also account for at least two subgroups.
In addition, to the two unique Sylow subgroups, we must also have at least $2q+1$ of the other two types.  Together with the at least 9 product subgroups, $G$, and $\{e\}$, this is already $2q+14 \geq 20$ subgroups.  A similar argument assuming a different pair are normal will be identical except that there will be strictly more than $2q+1$ total Sylow subgroups.

(4) Suppose that three of them are not normal.  If $\Syl{S} {\not}{\unlhd} G$, then there are at least $2q+s+2$ Sylow subgroups.  Together with the three products involving the lone normal Sylow subgroup and $G$ and $\{e\}$ this is already $2q+s+7 \geq 20$ subgroups.  If instead, $\Syl{S} \unlhd G$, then  for each product subgroup $H_{p^as^d}$, $H_{q^bs^d}$, and $H_{r^cs^d}$, observe that either a single product subgroup contains all of the non-normal Sylow subgroups for the associated prime (implying that $n_p$, $n_q$, or $n_r \geq s$), or there must be multiple subgroups of that order (at least $q$, $p$, or $p$ respectively).  Together with $G$ and $\{e\}$ this is already at least $2q+r+10 \geq 21$ subgroups.

(5) Finally, if all four of those Sylow subgroups are not normal, then we must have at least $2q+r+s+3$ of them in total.  Together with $G$ and $\{e\}$ this is already $2q+r+s+5 \geq 23$ subgroups.
\end{proof}

\subsection{Classifying non-abelian groups with $|\Sub G| = k$ for $k \leq 19$}
Theorems~\ref{thm:p^aq^b}, \ref{thm:p^aq^br^c}, \ref{thm:pqr}, and \ref{thm:4primes} together reduce the search space for non-abelian groups with 19 or fewer subgroups.  Specifically, any non-abelian group $G$ with $k \leq 19$ subgroups must either be a direct product of groups with coprime orders (whose individual subgroup counts multiply to $k$), or $G$ must be non-nilpotent with $|G|$ satisfying one of the options listed in Table~\ref{tab:options}.
\begin{table}[h]
\centering
\begin{tabular}{c||c||c}
$p^7q$ with $q=3$ & $p^2q^4$ 
with $q=3$ & $2^i$ with $i \leq 6$\\
$p^6q$ with $q \leq 5$ & $p^2q^3$ with $q \leq 5$ & $3^i$ with $i \leq 5$ \\
$p^5q$ or $p^4q$ with $q \leq 7$ & $p^3q^2$ with $p=2$ and $q\leq 7$ & $5^i$ with $i \leq 3$ \\
$p^3q$ with $q\leq 11$ & $p^2q^2$ with $p=2,3$ and $q\leq 7$ 
& $7^i$ with $i \leq 3$  \\
$p^2q$ or $pq$ with $q\leq 13$ & $p^3q^3$ with $q=3$ \\
$pq^4$ or $pq^3$ with $q\leq 3$ & $pqr$ with $r\leq 7$\\
$pq^2$ with $q \leq 7$ & $p^2qr$ or $pq^2r$ or $pqr^2$ with $r=5$ \\

\end{tabular}
\caption{Potential values for $|G|$ if $|\Sub G| \leq 19$.}
\label{tab:options}
\end{table}

With this reduced search space, we can use GAP to search systematically beginning with smaller numbers of subgroups working up.  Of course, when $k$ is smaller than 19, we need not consider the entire search space, but we can avoid duplicating effort in this way. 
Table~\ref{Table: Non-Abelian Groups} lists the similarity classes of all non-abelian groups with 19 or fewer subgroups. Note we omit empty rows and, as before, any arbitrary primes that appear are assumed to be coprime to the others.

\begin{table}[h]
    \centering
    \begin{tabular}{c|c|c}
         $|\Sub G|$ & Similarity Classes & \# of \\
         & of Non-abelian Groups & Classes \\ \hline
        $6$ &  $Q_8$, $S_3$ & 2 \\  \hline
        $8$ &  $\Dic_{12}$, $D_{10}$ & 2  \\  \hline
        $10$  & $\Z_7 \rtimes \Z_3$, $\Z_3 \rtimes \Z_8$, $D_8$, $D_{14}$, $M_{27}$, $\Dic_{20}$,
        $A_4$ & 7   \\  \hline
         $11$ & $Q_{16}$, $M_{16}$ & 2  \\  \hline
        $12$ & $Q_8 \times \Z_p$, $S_3 \times \Z_p$, $\Z_3 \rtimes \Z_{16}$, $\Dic_{28}$, 
        $\Z_7 \rtimes \Z_9$, $\Z_5 \rtimes \Z_8$ & 6   \\  \hline
        $14$ & $M_{32}$, $S_3 \times \Z_3$, $\Z_3 \rtimes \Z_{32}$, $\Z_5 \rtimes \Z_{16}$, $\operatorname{GA}(1,5)$, $\Z_7 \rtimes \Z_8$, $D_{22}$, & 11\\
        & $\Z_{27} \rtimes \Z_3$, $\Z_7 \rtimes \Z_{27}$,  $\Z_{11} \rtimes \Z_5$,  $\Z_{25} \rtimes \Z_5$ \\ \hline
        $15$ &  $\operatorname{SL}(2,3), \operatorname{SD}_{16},$ $\Z_4 \rtimes \Z_4$, $(\Z_2 \times \Z_2) \rtimes \Z_9$ & 4\\ \hline
        16 & $\Dic_{12} \times \Z_p$, $D_{10} \times \Z_p$, $D_{18}$, $D_{12}$, $\Z_5 \rtimes \Z_8$, $\Z_5 \rtimes \Z_{32}$, $\Z_3 \rtimes \Z_{64}$ & 13\\
        & $\Z_7 \rtimes \Z_{16}$, $\Dic_{44}$, 
        $D_{26}$, $\Z_{13} \rtimes \Z_3$, $\Z_7 \rtimes \Z_{81}$, $\Z_{11} \rtimes \Z_{25}$  \\ \hline
        17 & $\Z_{32} \rtimes \Z_2$ & 1\\ \hline
        18 & $Q_8 \times \Z_{p^2}$, $S_3 \times \Z_{p^2}$, $\Z_8.\Z_4$, $\Z_3 \rtimes \Z_{128}$, $\Dic_{18}$, 
        $\Z_5 \rtimes \Z_{16}$, $\Z_{81} \rtimes \Z_3$,   $\Z_7 \rtimes \Z_{243}$ & 15\\
        & $\Z_{49} \rtimes \Z_7$, $\Z_{13} \rtimes \Z_9$, $\Dic_{52}$, 
        $\Z_{11} \rtimes \Z_{125}$, $\Z_{11} \rtimes \Z_8$, $\Z_7 \rtimes \Z_{32}$, $\Z_5 \rtimes \Z_{64}$ \\ \hline
        19 & $\Z_2 \times Q_8$, $D_{16}$, $(\Z_3 \times \Z_3) \rtimes \Z_3$, $\Dic_{36}$
        & 4
    \end{tabular}
    \caption{Non-Abelian Groups with $19$ or fewer Subgroups}
    \label{Table: Non-Abelian Groups}
\end{table}

To further clarify, $\Dic_{n}$ represents the dicyclic group of order $n$, and we now describe the specific semi-direct products and non-split extensions listed above in terms of generators and relations -- but only when they are not unique. 
With 12 subgroups, $\Z_5 \rtimes \Z_8 = \langle x,y \mid x^5=y^8=e, yxy^{-1} = x^{-1}\rangle$.  With 14 subgroups, $\Z_5 \rtimes \Z_{16} = \langle x,y \mid x^5=y^{16}=e, yxy^{-1} = x^{-1}\rangle$ and $\Z_{25} \rtimes \Z_5 = \langle x,y \mid x^{25}=y^5=e, yxy^{-1} = x^6\rangle$.  With 16 subgroups, $\Z_5 \rtimes \Z_8 = \langle x,y \mid x^5=y^8=e, yxy^{-1} = x^3\rangle$ and $\Z_5 \rtimes \Z_{32} = \langle x,y \mid x^5=y^32=e, yxy^{-1} = x^{-1}\rangle$.  With 17 subgroups $\Z_{32} \rtimes \Z_2 = \langle x,y \mid x^32 = y^2 = e, yxy^{-1} = x^{17}\rangle$. With 18 subgroups $\Z_8.\Z_4 = \langle x,y \mid x^8=e, x^4=y^4, yxy^{-1} = x^{-1}\rangle$, $\Z_3 \rtimes \Z_{128} = \langle x,y \mid x^3 = y^{128} = e, yxy^{-1} = x^{-1}\rangle$, $\Z_5 \rtimes \Z_{16} = \langle x,y \mid x^5=y^{16}=e, yxy^{-1} = x^3\rangle$, and $\Z_5 \rtimes \Z_{64} = \langle x,y \mid x^5 = y^{64} = e, yxy^{-1} = x^{-1}\rangle$.

\begin{remark}[$|\Sub G|$ prime]
At first, the authors suspected that a non-abelian group with a prime number of subgroups would have to be a $p$-group (as is true for abelian groups).  Discovering the counterexample of $A_5$ -- which has 59 subgroups -- we adjusted this conjecture to perhaps \emph{solvable} non-abelian groups.  However, as seen above, $Z_9 \rtimes \Z_4$ has 19 subgroups despite being solvable and not a $p$-group.
\end{remark}

Given the classification of both abelian and non-abelian groups with 19 or fewer subgroups, we can combine these results to give the first 19 terms in sequence A274847 \cite{OEIS}.  Those terms are 1, 1, 1, 2, 2, 5, 1, 7, 2, 12, 4, 11, 1, 17, 8, 22, 3, 22, 5.  (Once again, the 10th term should be 12, not 11.)  Further exploration using these techniques is possible, however, as we are already running into the limits of what GAP can check, it would require improving some of the bounding arguments made above.  One place ripe for improvement is the case of $p^aq^b$ with $a,b \geq 2$.  There are no such non-nilpotent groups with fewer than 20 subgroups when $q>3$ even though our bounds did not rule some of those cases out.


\begin{thebibliography}{}

\bibitem{Abelian Sub thesis}{Wafa Omar Ali and  Mohamed Al-Awami Ali}, The Number of Subgroups of a Finite Abelian Group, University of Benghazi, \url{http://repository.uob.edu.ly/bitstream/handle/123456789/317}

\bibitem{AM}{Stefanos Aivazidis and Thomas M\"uller}, Finite non-cyclic $p$-groups whose number of subgroups is minimal, \emph{Archiv der Mathematik}, \textbf{114} (2020) pp 13--17. \url{https://doi.org/10.1007/s00013-019-01376-9}

\bibitem{BN} Alexander Betz and David A. Nash, A note on abelian groups with fewer than 50 subgroups, \emph{preprint}, (2020).

\bibitem{Burnside} William Burnside, Notes on the theory of groups of finite order, \emph{Proc. London Math. Soc.}, \textbf{26} (1895) pp 191--214.

\bibitem{Cole} {F.\ N. Cole and J. W.\ Glover}, On groups whose orders are products of three prime factors, \emph{Amer. J. Math.}, \textbf{15} No. 3, (1893) pp 191--220. \url{www.jstor.org/stable/2369839} 

\bibitem{Lam} {T.\ Y. Lam}, On subgroups of prime index, \emph{Amer. Math. Monthly}, \textbf{111} No. 3, (Mar 2004) pp 256--258. \url{https://doi.org/10.2307/4145135}

\bibitem{Miller1} {G.\ A. Miller}, Groups having a small number of subgroups, \emph{Proc. Natl. Acad. Sci. USA}, \textbf{25} No. 7, (July 1939) pp 367--371. \url{https://doi.org/10.1073/pnas.25.7.367}

\bibitem{Miller2} {G.\ A. Miller}, Groups which contain less than ten proper subgroups, \emph{Proc. Natl. Acad. Sci. USA}, \textbf{25} No. 9, (Sept 1939) pp 482--485. \url{https://doi.org/10.1073/pnas.25.9.482}

\bibitem{Miller3} {G.\ A. Miller}, Groups which contain ten or eleven proper subgroups, \emph{Proc. Natl. Acad. Sci. USA}, \textbf{25} No. 10, (Oct 1939) pp 540--543. \url{https://doi.org/10.1073/pnas.25.10.540}

\bibitem{Miller4} {G.\ A. Miller}, Groups which contain less than fourteen proper subgroups, \emph{Proc. Natl. Acad. Sci. USA}, \textbf{26} No. 2, (Feb 1940) pp 129--132. \url{https://doi.org/10.1073/pnas.26.2.129}

\bibitem{Miller5} {G.\ A. Miller}, Groups which contain exactly fourteen proper subgroups, \emph{Proc. Natl. Acad. Sci. USA}, \textbf{26} No. 4, (Apr 1940) pp 283--286. \url{https://doi.org/10.1073/pnas.26.4.283}

\bibitem{Nash} {David A. Nash}, A Friendly Introduction to Group Theory, 2nd Edition,  CreateSpace, Seattle, 2017.

\bibitem{Slattery} {Michael C. Slattery}, Groups with at most twelve subgroups, \emph{preprint} (2016) \url{https://arxiv.org/pdf/1607.01834.pdf}

\bibitem{OEIS} {N.J.A.\ Sloane}, Sequence A274847, \emph{The Online Encyclopedia of Integer Sequences}, \url{https://oeis.org/A274847}

\bibitem{MT Abelian Sub}{Marius T\u{a}rn\u{a}uceanu}, Counting subgroups for a class of finite nonabelian, \emph{An.\ Univ.\ Vest Timi\c{s}.\ Ser.\ Mat.-Inform.}, \textbf{46} No. 1, (2008) pp 145--150. 

\bibitem{Wielandt} Helmut Wielandt, Ein Beweis f\"{u}r die Existenz der Sylowgruppen. \emph{Arch. Math} \textbf{10} (1959) pp 401--402. \url{https://doi.org/10.1007/BF01240818}





\end{thebibliography}
\end{document}